\documentclass[11pt]{article}
\usepackage[utf8]{inputenc}
\usepackage[T1]{fontenc}
\usepackage{comment}
\usepackage{authblk}
\usepackage{relsize}
\usepackage{changepage}
\usepackage{mathtools}
\usepackage[english]{babel}
\usepackage{amsmath}
\usepackage{amsthm}
\usepackage{amsfonts}
\usepackage{amssymb}
\usepackage{graphicx}
\usepackage[usenames,dvipsnames]{color}
\theoremstyle{plain}
\newtheorem{theorem}{Theorem}[section]
\newtheorem{corollary}[theorem]{Corollary}

\newtheorem{fact}[theorem]{Fact}
\newtheorem{question}[theorem]{Question}

\newtheorem{claim}[theorem]{Claim}
\newtheorem{observation}[theorem]{Observation}

\makeatletter
\newcommand{\vast}{\bBigg@{4}}
\newcommand{\Vast}{\bBigg@{5}}
\makeatother
\definecolor{bulgarianrose}{rgb}{0.28, 0.02, 0.03}
\definecolor{gray}{rgb}{0.5, 0.5, 0.5}

\theoremstyle{definition}

\theoremstyle{remark}
\newtheorem*{remark}{Remark}
\newtheorem*{fact*}{Fact}
\newtheorem*{question*}{Question}
\usepackage[left=2cm,right=2cm,top=2cm,bottom=2cm]{geometry}
\usepackage{amssymb}
\usepackage{hyperref}
\makeatletter
\def\namedlabel#1#2{\begingroup
    #2%
    \def\@currentlabel{#2}%
    \phantomsection\label{#1}\endgroup
}
\usepackage[normalem]{ulem}
\usepackage{dsfont}
\usepackage{accents}
\usepackage{verbatim}
\usepackage{ulem}
\usepackage{pgf,tikz,pgfplots}
\usepackage[all]{xy} 

\usepackage{stackengine}
\stackMath
\newcommand\tsup[2][2]{%
 \def\useanchorwidth{T}%
  \ifnum#1>1%
    \stackon[-.5pt]{\tsup[\numexpr#1-1\relax]{#2}}{\scriptscriptstyle\sim}%
  \else%
    \stackon[.5pt]{#2}{\scriptscriptstyle\sim}%
  \fi%
}

\newcommand{\towr}{\text{towr}}

\usepackage{subfig}

\usepackage{enumerate}

\title{\scshape
  A note on the Erd\H{o}s-Szekeres theorem in two dimensions}

\author{Lyuben Lichev}
\affil{Ecole Normale Supérieure de Lyon, Lyon, France}

\begin{document}

\maketitle
 
\begin{abstract}
In \cite{BM1, Kal} respectively Burkill and Mirsky, and Kalmanson, prove independently that, for every $r\ge 2, n\ge 1$, there is a sequence of $r^{2^n}$ vectors in $\mathbb R^n$, which does not contain a subsequence of $r+1$ vectors $v^1, v^2,\dots,v^{r+1}$ such that, for every $i$ between 1 and $n$, $(v^{j}_i)_{1\le j\le r+1}$ forms a monotone sequence. Moreover, $r^{2^n}$ is the largest integer with this property. In this short note, for two vectors $u = (u_1, u_2,\dots, u_n)$ and $v = (v_1, v_2, \dots, v_n)$ in $\mathbb R^n$, we say that $u\le v$ if, for every $i$ between 1 and $n$, $u_i\le v_i$. Just like Burkill and Mirsky, and Kalmanson, for every $k, \ell\ge 1, d\ge 2$ we find the maximal $N_1, N_2$ (which turn out to be equal) such that there are numerical two-dimensional arrays of size $(k+\ell-1)\times N_1$ and $(k+\ell)\times N_2$, which neither contain a subarray of size $k\times d$, whose columns form an increasing sequence of $d$ vectors in $\mathbb R^k$, nor contain a subarray of size $\ell\times d$, whose columns form a decreasing sequence of $d$ vectors in $\mathbb R^{\ell}$. In a consequent discussion, we consider a generalisation of this setting and make a connection with a famous problem in coding theory.
\end{abstract}

\hspace{1em}Keywords: Erd\H{o}s-Szekeres, two-dimensional array, monotone sequence of vectors

\hspace{1em}MSC Class: 05D10

\section{Introduction}
For every $n\ge 1$, we denote by $[n]$ the set of integers between 1 and $n$. For a sequence $(a_i)_{i=1}^{n}$ and any $k\in [n]$, a \emph{subsequence} of $(a_i)_{i=1}^{n}$ of length $k$ is any sequence $a_{i_1}, a_{i_2}, \dots, a_{i_k}$ with $1\le i_1 < i_2 < \dots < i_k\le n$. In this note, an \emph{array $A$ of size $m\times n$} is a rectangular matrix of size $m\times n$, and a \emph{subarray} of $A$ is obtained by intersecting a subset of the rows of $A$ with a subset of the columns of $A$.\par

The following theorem due to Erd\H{o}s and Szekeres is one of the cornerstones in the field of Ramsey theory. 

\begin{theorem}[Erd\H{o}s-Szekeres theorem, \cite{ES}]\label{ES thm}
For all integers $k, \ell\ge 1$, any sequence of $k\ell + 1$ real numbers contains either an increasing subsequence of length $k+1$ or a decreasing subsequence of length $\ell+1$.\footnote{In this paper, an increasing (respectively decreasing) sequence is not necessarily strictly increasing (respectively strictly decreasing).}\qed
\end{theorem}

Since the appearance of the theorem in 1935, a vast number of generalisations have been brought forward. A few of them can be found in \cite{BM2, BM1, Kal}. In each of the three papers, a major subject was the topic of finding large subarrays of numerical arrays, in which the entries of every row or every column, or both, form a monotone sequence. To be more precise, we state a version of Theorem 1 from \cite{Kal}, which also appears as Theorem 2.4 in \cite{BM1}:

\begin{theorem}[\cite{BM1, Kal}]\label{Kal thm}
Let $S$ be a sequence of vectors in $\mathbb R^n, (n \ge 1)$. If $S$ has length at least $r^{2^n} + 1$, then it contains a subsequence of $r+1$ vectors $v^1, v^2,\dots,v^{r+1}$ such that, for every $i$ between 1 and $n$, $(v^{j}_i)_{1\le j\le r+1}$ forms a monotone sequence. Moreover, $r^{2^n} + 1$ cannot be replaced by $r^{2^n}$.
\end{theorem}

In another line of research, the papers \cite{BST, FG} consider another notion of monotonicity. For two vectors $u = (u_1, u_2,\dots, u_n)$ and $v = (v_1, v_2, \dots, v_n)$ in $\mathbb R^n$, we say that $u\le v$ if, for every $i\in [n]$, $u_i\le v_i$. A sequence of vectors $v^1, v^2,\dots, v^k\in \mathbb R^d$ is said to be \emph{monotone} if $v^1\le v^2\le \dots\le v^k$ (in this case, the sequence is \emph{increasing}) or $v^k\le v^{k-1}\le \dots \le v^1$ (the sequence is then \emph{decreasing}). An array is called \emph{monotone} if both the set of its rows and the set of its columns form a monotone sequence of vectors. Perhaps the first to introduce this notion were Fishburn and Graham in \cite{FG}. The papers \cite{BST, FG} treat (a generalisation in dimension $d\ge 2$ of) the problem of finding the minimal $N$, for which every numerical array of size $N\times N$ contains a monotone subarray of size $n\times n$ (they denote this minimal $N$ by $M_2(n)$).\par 

Our setting interpolates between the weaker notion from \cite{BM2, BM1, Kal} and the stronger one from \cite{BST, FG}. In particular, we will be interested in the partial order used in \cite{BST, FG}, but only on the set of columns of an array and not on the set of rows.

\section{Our results}

We say that a numerical array has \emph{property $(k, \ell, d)$} if it contains either a subarray of size $k\times d$, whose columns form an increasing sequence of vectors in $\mathbb R^k$, or a subarray of size $\ell\times d$, whose columns form a decreasing sequence of vectors in $\mathbb R^{\ell}$.

\begin{theorem}\label{thm 1}
For every $k, \ell\ge 1, d\ge 2$, every array of size $(k+\ell-1)\times ((d-1)^{\binom{k+\ell}{k}}+1)$ has property $(k, \ell, d)$.
\end{theorem}
\begin{proof}
Our main tool will be the Erd\H{o}s-Szekeres theorem. We prove the theorem by induction on $k+\ell$. The case $k+\ell=2$ is true, since this is the statement of the Erd\H{o}s-Szekeres theorem itself. Suppose that the statement is true for every pair of positive integers of sum less than some integer $t\ge 3$. Let $k, \ell$ be positive integers with $k+\ell = t$. Consider the first row of an array of size $(k+\ell-1)\times ((d-1)^{\binom{k+\ell}{k}}+1)$. Then, by the Erd\H{o}s-Szekeres theorem, either it contains an increasing subsequence of size $(d-1)^{\binom{k+\ell-1}{k-1}}+1$, or it contains a decreasing subsequence of size $(d-1)^{\binom{k+\ell-1}{k}}+1$. If $k=1$ and the first case takes place, or if $\ell=1$ and the second case takes place, we are done. Otherwise, in both cases, it is sufficient to consider the subarray of the remaining $t-1$ rows and the columns, whose first entries participate in the monotone subsequence given above, to conclude by the induction hypothesis.
\end{proof}

\begin{remark}
For any $k, \ell\ge 1, M\ge 2$, Theorem~\ref{thm 1} cannot hold for an array of size $(k+\ell-2)\times M$. Indeed, letting the first $k-1$ rows be equal to $1,2,\dots, M$, and the remaining $\ell-1$ rows be equal to $M, M-1,\dots, 1$, there is no subarray of the form described in Theorem~\ref{thm 1}.\qed
\end{remark}

\begin{remark}
In \cite{FG} Fishburn and Graham gave an upper bound on $M_2(n)$ of the form $\towr_5(O(n))$, where the function $\towr_k$ is defined recursively as $\towr_1(n) = n, \towr_k(n) = 2^{\towr_{k-1}(n)}$. This bound was improved significantly by Buci\'c, Sudakov and Tran in \cite{BST} to
\begin{equation*}
    M_2(n)\le (2n)^{2^{2n}}.
\end{equation*}

Theorem~\ref{thm 1} may be seen as an improvement of Theorem 2.3 from \cite{BST} in the special case when $k = \ell = n$. This slightly improves the bound on $M_2(n)$ given there: we obtain 
\begin{equation*}
   M_2(n)\le 2\binom{(2n-2)^2+1}{2n-1} (n-1)^{\binom{2n}{n}} + 1.
\end{equation*}
Indeed, consider a numerical array of size $((2n-2)^2+1)\times \left(2\dbinom{(2n-2)^2+1}{2n-1} (n-1)^{\binom{2n}{n}} + 1\right)$. By the Erd\H{o}s-Szekeres theorem every column of the array contains a monotone subsequence of length $2n-1$. Since any such subsequence may be increasing or decreasing, and the number of possible choices for the indices of the rows of its entries is $\binom{(2n-2)^2+1}{2n-1}$, one may deduce by the Pigeonhole principle that there is a subarray of size $(2n-1)\times \left((n-1)^{\binom{2n}{n}} + 1\right)$, whose rows form a monotone sequence of vectors of dimension $\mathbb R^{(n-1)^{\binom{2n}{n}} + 1}$. By Theorem~\ref{thm 1} we find a monotone subarray of size $n\times n$ in the latter.\qed
\end{remark}

What is most surprising about Theorem~\ref{thm 1} is that it is indeed tight for any choice of parameters $k,\ell\ge 1, d\ge 2$.

\begin{theorem}\label{thm 2}
For every $k, \ell\ge 1, d\ge 2$, there is an array of size $(k+\ell)\times (d-1)^{\binom{k+\ell}{k}}$, which does not have property $(k, \ell, d)$.
\end{theorem}

Note that Theorem~\ref{thm 2} is slightly stronger than a purely complementary statement of Theorem~\ref{thm 1} since we consider arrays of size $(k+\ell)\times (d-1)^{\binom{k+\ell}{k}}$ and not $(k+\ell-1)\times (d-1)^{\binom{k+\ell}{k}}$.\par 

Before constructing an array of size $(k+\ell)\times (d-1)^{\binom{k+\ell}{k}}$, which does not have property $(k, \ell, d)$, we make an elementary observation. A walk in $\mathbb Z^2$ is \emph{up-right} if it uses only transitions of the type $(x,y)\to (x+1,y)$ and $(x,y)\to (x, y+1)$.
\begin{observation}\label{trivial}
For every $k,\ell\in \mathbb N$, the number of up-right walks in $\mathbb Z^2$ from $(0,0)$ to $(k, \ell)$ is $\binom{k+\ell}{k}$.\qed
\end{observation}

For technical reasons, let us rotate the grid $\mathbb Z^2$ at $-45^{\circ}$ and consider walks from $(0,0)$ to $\left(\frac{k+\ell}{\sqrt{2}}, \frac{k-\ell}{\sqrt{2}}\right)$ (which is the image of $(k, \ell)$ under the rotation) instead. Each of these walks contains $k$ steps of the type $\nearrow$ and $\ell$ steps of the type $\searrow$. By Observation~\ref{trivial} there are $\binom{k+\ell}{k}$ such walks, which we list in an arbitrary order.\par 

Let $A$ be an array of size $(k+\ell)\times \binom{k+\ell}{k}$ with entries in the set $\{\nearrow, \searrow\}$, for which the $j-$th column $(A_{i,j})_{i=1}^{k+\ell}$ describes the $j-$th walk in the list above (one begins the walk with the arrow $A_{1,j}$). See Figure~\ref{fig 1}.

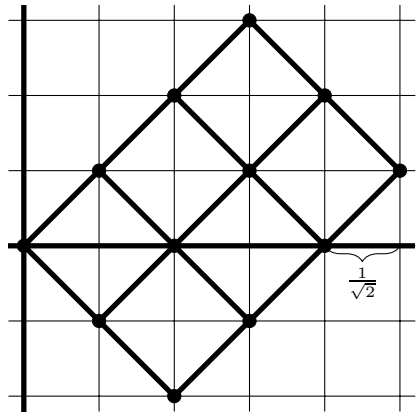
\begin{figure}[h]
\centering
\begin{tikzpicture}[line cap=round,line join=round,x=1cm,y=1cm]
\draw [xstep=1cm,ystep=1cm] (-4.2,-2.2) grid (1.2,3.2);
\clip(-4.2,-2.2) rectangle (1.2,3.2);
\draw [line width=2pt] (-4,-6.04) -- (-4,6.04);
\draw [line width=2pt,domain=-12.8:12.8] plot(\x,{(-0-0*\x)/2});
\draw [line width=2pt] (-4,0)-- (-3,1);
\draw [line width=2pt] (-3,1)-- (-2,2);
\draw [line width=2pt] (-2,2)-- (-1,3);
\draw [line width=2pt] (-1,3)-- (0,2);
\draw [line width=2pt] (0,2)-- (1,1);
\draw [line width=2pt] (-4,0)-- (-3,-1);
\draw [line width=2pt] (-3,-1)-- (-2,-2);
\draw [line width=2pt] (-2,-2)-- (-1,-1);
\draw [line width=2pt] (-1,-1)-- (0,0);
\draw [line width=2pt] (0,0)-- (1,1);
\draw [line width=2pt] (-2,2)-- (-1,1);
\draw [line width=2pt] (-1,1)-- (0,0);
\draw [line width=2pt] (0,2)-- (-1,1);
\draw [line width=2pt] (-3,1)-- (-2,0);
\draw [line width=2pt] (-2,0)-- (-1,-1);
\draw [line width=2pt] (-3,-1)-- (-2,0);
\draw [line width=2pt] (-2,0)-- (-1,1);

\draw [decorate,decoration={brace,amplitude=6pt},xshift=0pt,yshift=0pt]
(1,0) -- (0,0) node {};
\begin{scriptsize}
\draw [fill=black] (-4,0) circle (2.5pt);
\draw [fill=black] (1,1) circle (2.5pt);
\draw [fill=black] (-3,1) circle (2.5pt);
\draw [fill=black] (-2,2) circle (2.5pt);
\draw [fill=black] (-1,3) circle (2.5pt);
\draw [fill=black] (0,2) circle (2.5pt);
\draw [fill=black] (-3,-1) circle (2.5pt);
\draw [fill=black] (-2,-2) circle (2.5pt);
\draw [fill=black] (-1,-1) circle (2.5pt);
\draw [fill=black] (0,0) circle (2.5pt);
\draw [fill=black] (-1,1) circle (2.5pt);
\draw [fill=black] (-2,0) circle (2.5pt);
\draw [fill=black] (0.5,-0.5) node {{$\frac{1}{\sqrt{2}}$}};
\end{scriptsize}
\end{tikzpicture}
\captionsetup{singlelinecheck=off}
\caption[]{The figure depicts the example $k=3, \ell=2$. One possibility for the array $A$ is given by 
\begin{displaymath}
\begin{bmatrix}
\nearrow & \nearrow & \nearrow & \nearrow & \nearrow & \nearrow & \searrow & \searrow & \searrow & \searrow \\

\nearrow & \nearrow & \nearrow & \searrow & \searrow & \searrow & \nearrow & \nearrow & \nearrow & \searrow \\

\nearrow & \searrow & \searrow & \nearrow & \nearrow & \searrow & \nearrow & \nearrow & \searrow & \nearrow \\

\searrow & \nearrow & \searrow & \nearrow & \searrow & \nearrow & \nearrow & \searrow & \nearrow & \nearrow \\

\searrow & \searrow & \nearrow & \searrow & \nearrow & \nearrow & \searrow & \nearrow & \nearrow & \nearrow
\end{bmatrix}.
\end{displaymath}}
\label{fig 1}
\end{figure}

Now, we show how to relate the array $A$ to the construction of a numerical array, satisfying the requirements of Theorem~\ref{thm 2}. In the beginning, let $B$ be an empty array of size $(k+\ell)\times (d-1)^{\binom{k+\ell}{k}}$. We will fill in the entries of $B$ so that each row forms a permutation of the integers from 1 to $(d-1)^{\binom{k+\ell}{k}}$. For every $i$ between 0 and $\binom{k+\ell}{k}$ and for every $j\in [k+\ell]$, divide the $j-$th row of $B$ into $(d-1)^{\binom{k+\ell}{k}-i}$ groups of $(d-1)^i$ consecutive entries. For each group, we replace its entries in the $j-$th row of $B$ with the consecutive integers in one of the sets 
\begin{equation*}
[1, (d-1)^i], [(d-1)^i+1, 2(d-1)^i], \dots, [(d-1)^{\binom{k+\ell}{k}}-(d-1)^i+1, (d-1)^{\binom{k+\ell}{k}}].
\end{equation*}

Since for every $i\in [\binom{k+\ell}{k}]$ we require that the integers from any group of size $(d-1)^i$ in the $j-$th row of $B$ are consecutive, we may consider an order relation between the groups, defined by $g_1 < g_2$ if every entry in the group $g_1$ is smaller than every entry in the group $g_2$. Now, we will use the $j-$th row of the array $A$ to indicate the order of the integers in every group in the $j-$th row of $B$, and the order of the groups as well, as follows. For every row $j\in [k+\ell]$ and for every $i\in [\binom{k+\ell}{k}]$, if $A_{j,i} = \nearrow$, then for every group $g_i$ containing $(d-1)^i$ consecutive entries in the $j-$th row of $B$, we order the $d-1$ groups of $(d-1)^{i-1}$ consecutive entries, contained in $g_i$, in increasing order. If, on the other hand, $A_{j,i} = \searrow$, then for every group $g_i$ of $(d-1)^i$ entries in the $j-$th row of $B$, we order the $d-1$ groups of $(d-1)^{i-1}$ consecutive entries, contained in $g_i$, in decreasing order. For the sake of clarity, we give an example.\par 

For example, consider the case $k=2, \ell=1, d=4$. One possibility for the array $A$ is 
\begin{equation*}
    \begin{bmatrix}
    \nearrow & \nearrow & \searrow \\
    \nearrow & \searrow & \nearrow \\
    \searrow & \nearrow & \nearrow
    \end{bmatrix}.
\end{equation*}

Since $(d-1)^{\binom{k+\ell}{k}} = 3^3 = 27$, in the beginning we are given an empty array $B$ of size $3\times 27$. Let us see how to fill in the first row. We have $A_{1,3} = \searrow$, so this means that the first row in $B$ must begin with the group of $(d-1)^{\binom{k+\ell}{k}-1} = 3^2 = 9$ integers from 19 to 27 in some order, then continue with the group of nine integers from 10 to 18 in some order, and finish with the group of nine integers from 1 to 9 in some order. Then, $A_{1,2} = \nearrow$, so each of the three groups above will consist of three smaller groups, ordered in increasing order. For example, the integers in the first group will be 19, 20, 21 in some order, then we continue with 22, 23, 24 in some order, and we finish with 25, 26, 27 in some order. Also, the second group of nine integers will begin with 10, 11, 12 in some order, continue with 13, 14, 15 in some order, and then finish with 16, 17, 18 in some order. Finally, $A_{1,1} = \nearrow$, so the members of each of the nine groups of three consecutive integers will be ordered in an increasing order. Therefore, the first row of $B$ is given by
\begin{equation*}
19,20,21,22,23,24,25,26,27,10,11,12,13,14,15,16,17,18,1,2,3,4,5,6,7,8,9.
\end{equation*}

Let us analyse the second row. Since $A_{2,3}=\nearrow$, it will begin with the integers from 1 to 9 in some order, then it will continue with the integers from 10 to 18 in some order, and will finish with the integers from 19 to 27 in some order. Then, $A_{2,2}=\searrow$, so the first group of nine integers will begin with 7,8,9 in some order, then it will continue with 4,5,6 in some order, and will finish with 1,2,3 in some order. The two other groups of nine integers are organised similarly, that is, $16,17,18; 13,14,15; 10,11,12$ for the second one and $25,26,27; 22,23,24; 19,20,21$ for the third one. Finally, $A_{2,1}=\nearrow$, so each group of three consecutive integers is ordered in an increasing order. The second row of $B$ is thus given by
\begin{equation*}
7,8,9,4,5,6,1,2,3,16,17,18,13,14,15,10,11,12,25,26,27,22,23,24,19,20,21.
\end{equation*}
The third row of $B$ is filled in in a similar way:
\begin{equation*}
3,2,1,6,5,4,9,8,7,12,11,10,15,14,13,18,17,16,21,20,19,24,23,22,27,26,25.
\end{equation*}

We prove that the array $B$ does not have property $(k, \ell, d)$.

\begin{proof}[Proof of Theorem~\ref{thm 2}]
We argue by contradiction. Suppose that there exists a subarray $D$ of size $k\times d$, whose columns form an increasing sequence of $k-$dimensional vectors (the case of a subarray $D$ of size $\ell\times d$, whose columns form a decreasing sequence of $\ell-$dimensional vectors is analogous). Let $c_1, c_2,\dots, c_d$ be the indices of the columns of $B$, which contain the entries of $D$. 
\begin{claim}\label{sublem}
There exist two integers $i_1, i_2\in [\binom{k+\ell}{k}]$, for which, for both $i = i_1, i_2$, there is a group of size $(d-1)^i$, which contains at least two groups of size $(d-1)^{i-1}$, all of which contain at least one of the columns with indices $c_1, c_2, \dots, c_d$.
\end{claim}
\begin{proof}
Let $i_1$ be the largest $i$ between $1$ and $\binom{k+\ell}{k}$, for which the set $c_1, c_2, \dots, c_d$ has elements in at least two different groups of size $(d-1)^{i_1-1}$. Note that $i_1$ is at least two since a group of size $d-1$ cannot contain all $d$ indices $c_1, \dots, c_d$. Moreover, by the pigeonhole principle, there is a group of size $(d-1)^{i_1-1}$, which contains at least two indices among $c_1,\dots, c_d$. For these two indices, there exists an integer $i_2 < i_1$, for which both of them are contained in a group of size $(d-1)^{i_2}$, but at the same time they are contained in different groups of size $(d-1)^{i_2-1}$ (note that the smallest groups are of size one). This proves the claim.
\end{proof}

Now, let $i_1$ and $i_2$ be the integers given by Claim~\ref{sublem}. Then, one must have that in each of the $k$ rows of $B$, which contain entries of $D$, the order of the groups of size $(d-1)^{i_1-1}$ within any group of size $(d-1)^{i_1}$ must be increasing, and the order of the groups of size $(d-1)^{i_2-1}$ within any group of size $(d-1)^{i_2}$ must also be increasing. This means that in the array $A$ must contain a subarray of size $k\times 2$, full of $\nearrow$. However, this is not possible since the columns of the array $A$ correspond to walks from $(0,0)$ to $(\frac{k+\ell}{\sqrt{2}}, \frac{k-\ell}{\sqrt{2}})$ with $k$ steps of the type $\nearrow$ and $\ell$ steps of the type $\searrow$, and the positions of the steps of the type $\nearrow$ determine a unique such walk. This is a contradiction, which proves that the array $B$ has property $(k, \ell, d)$.
\end{proof}

\section{Discussion}

A natural question to ask in light of our results up to now is: for $t\ge 1$, which is the largest integer $N = N(t)$, for which there exists a numerical array of size $(k+\ell+t)\times N$, which does not have property $(k, \ell, d)$? The idea from the proof of Theorem~\ref{thm 2} gives a clue for a lower bound. One could try to find a maximal family of paths with steps within $\{\nearrow, \searrow\}^{k+\ell+t}$, for which, at any $k$ positions among the $k+\ell+t$, at most one path of the family has only steps of the type $\nearrow$, and at any $\ell$ positions among the $k+\ell+t$, at most one path of the family has only steps of the type $\searrow$. The question might be reformulated as follows:

\begin{question}
What is the largest possible $N = N(k,\ell, t)\in \mathbb N$, for which there exists a family of $N$ binary vectors of length $k+\ell+t$ such that no two vectors in the family share $k$ entries equal to 1 ($\nearrow$), and no two vectors in the family share $\ell$ entries equal to 0 ($\searrow$)?
\end{question}

Arguing as in the proof of Theorem~\ref{thm 2}, we conclude the following fact:
\begin{fact}\label{fact 2}
For every $k, \ell, t\ge 1, d\ge 2$, there is an array of size $(k+\ell+t)\times (d-1)^{N(k, \ell, t)}$, which does not have property $(k, \ell, d)$.
\end{fact}

Sadly, we cannot compute $N(k,\ell,t)$ in general. In the case when the number of $1-$s (or $\nearrow-$s) is fixed for all the vectors in the family considered in the question above (say, to $k+t_1$ with $t_1$ between 0 and $t$, and let $t_2=t-t_1$), the problem comes down to looking for a family of paths between $(0,0)$ and $(\frac{k+\ell+t}{\sqrt{2}}, \frac{k+t_1-\ell-t_2}{\sqrt{2}})$ using only steps in the set $\{\nearrow, \searrow\}$, just like in the proof of Theorem~\ref{thm 2}. In the language of coding theory, we are looking for the largest constant-weight code of length $k+\ell+t$, distance $2\max(1+t_1, 1+t_2)$ and weight $k+t_1$. It is a famous and vastly studied problem in coding theory to compute this number, often denoted by $A(k+\ell+t, 2\max(1+t_1, 1+t_2), k+t_1)$, see \cite{Bro, MS, SHP} for a number of lower and upper bounds on the function $A(\cdot, \cdot, \cdot)$ (not to be confused with the array $A$ constructed above). Now, let 
\begin{equation*}
    M(k, \ell, t) = \max_{0\le t_1\le t} A(k+\ell+t, 2\max(1+t_1, 1+t_2), k+t_1).
\end{equation*}
Since $A(k+\ell+t, 2\max(1+t_1, 1+t_2), k+t_1)\le N(k, \ell, t)$ for every choice of $t_1$ between 0 and $t$, we conclude that $M(k, \ell, t)\le N(k, \ell, t)$ as well. As a consequence we obtain following corollary. We believe that it may be used in practice in the construction of large numerical arrays satisfying property $(k, \ell, d)$ for some small values of the parameter:
\begin{corollary}
For every $k, \ell, t\ge 1$, the construction from the proof of Theorem~\ref{thm 2} yields an array of size $(k+\ell+t)\times (d-1)^{M(k, \ell, t)}$, which does not have property $(k, \ell, d)$.
\end{corollary}

Reasonable upper bounds in this more general case may also be of interest since our approach for Theorem~\ref{thm 1} breaks down in case more than $k+\ell-1$ rows are present.

\section{Acknowledgements}
The author would like to thank Emil Kolev for useful remarks on the discussion section, Violeta Naydenova for a meticulous proofreading, and Peter Boyvalenkov, Ivailo Hartarsky, and the anonymous referee for several corrections.

\bibliographystyle{plain}
\bibliography{References}

\begin{thebibliography}{1}

\bibitem{Bro}
A.E.Brouwer.
\newblock Bounds for binary constant weight codes.
\newblock \url{https://www.win.tue.nl/~aeb/codes/Andw.html}.

\bibitem{BST}
M.~Buci\'c, B.~Sudakov, and T.~Tran.
\newblock Erd{\H o}s-{S}zekeres theorem for multidimensional arrays, 2019.

\bibitem{BM2}
H.~Burkill and L.~Mirsky.
\newblock Combinatorial problems on the existence of large submatrices {I}.
\newblock {\em Discrete Mathematics}, 6(1):15 -- 28, 1973.

\bibitem{BM1}
H.~Burkill and L.~Mirsky.
\newblock Monotonicity.
\newblock {\em Journal of Mathematical Analysis and Applications},
  41(2):391--410, 1973.

\bibitem{ES}
P.~Erd\H{o}s and G.~Szekeres.
\newblock A combinatorial problem in geometry.
\newblock {\em Compositio mathematica}, 2:463--470, 1935.

\bibitem{FG}
P.~C. Fishburn and R.~L. Graham.
\newblock Lexicographic {R}amsey theory.
\newblock {\em Journal of Combinatorial Theory, Series A}, 62(2):280--298,
  1993.

\bibitem{Kal}
K.~Kalmanson.
\newblock On a theorem of {E}rd{\H o}s and {S}zekeres.
\newblock {\em Journal of Combinatorial Theory, Series A}, 15(3):343--346,
  1973.

\bibitem{MS}
F.~J. MacWilliams and N.~J.~A. Sloane.
\newblock {\em The theory of error correcting codes}, volume~16.
\newblock Elsevier, 1977.

\bibitem{SHP}
D.~H. Smith, L.~A. Hughes, and S.~Perkins.
\newblock A new table of constant weight codes of length greater than 28.
\newblock {\em The electronic journal of combinatorics}, 13(1):A2, 2006.

\end{thebibliography}

\end{document}